\documentclass[11pt]{amsart}
\usepackage{amsmath,amsfonts,amssymb,amsthm,amscd,latexsym,euscript,verbatim,xypic}
\usepackage[all]{xy}

\makeatletter
\def\section{\@startsection{section}{1}%
  \z@{1.1\linespacing\@plus\linespacing}{.8\linespacing}%
  {\normalfont\Large\scshape\centering}}
\makeatother

\theoremstyle{plain}

\newtheorem*{conj*}{Root Groups Conjecture}
\newtheorem*{thm1.2}{(1.2) Theorem}
\newtheorem*{thm1.3}{(1.3) Theorem}
\newtheorem*{thm1.4}{(1.4) Theorem}
\newtheorem*{prop*}{Proposition}

\newtheorem{prop}{Proposition}[section]

\newtheorem{thm}[prop]{Theorem}
\newtheorem{cor}[prop]{Corollary}
\newtheorem{lemma}[prop]{Lemma}

\theoremstyle{definition}

\newtheorem*{Def*}{Definition}

\newtheorem{notation}[prop]{Notation}
\newtheorem*{notation*}{Notation}
\newtheorem{remark}[prop]{Remark}

\DeclareMathOperator{\tld}{\sim\!}

\newcommand{\rr}{\mathbb{R}}

\newcommand{\zz}{\mathbb{Z}}

\newcommand{\frakG}{\mathfrak{G}}

\newcommand{\frakM}{\mathfrak{M}}

\newcommand{\frakX}{\mathfrak{X}}

\newcommand{\gc}{\gamma}
\newcommand{\gC}{\Gamma}

\newcommand{\gd}{\delta}

\newcommand{\nsg}{\trianglelefteq}

\newcommand{\Aut}{{\rm Aut}}

\newcommand{\bu}{\bullet}

\numberwithin{equation}{section}

\begin{document}
\title[Crossed modules and topological groups]{Crossed modules as maps between connected components  of  topological groups}
\author[Emmanuel D.~Farjoun,\qquad Yoav Segev]
{Emmanuel~D.~Farjoun\qquad\quad Yoav Segev}
\address{Emmanuel D.~Farjoun\\
        Department of Mathematics\\
        Hebrew University of Jerusalem, Givat Ram\\
        Jerusalem 91904\\
        Israel}
\email{farjoun@math.huji.ac.il}

\address{Yoav Segev\\
        Department of Mathematics\\
        Ben Gurion University of the Negev\\
        Beer Sheva\\
        Israel}
\email{yoavs@math.bgu.ac.il}


\begin{abstract}
The purpose of this note is to observe that a homomorphism of discrete groups $f:\Gamma\to G$  arises as
the induced map $\pi_0(\mathfrak{M})\to \pi_0(\mathfrak{X})$ on path components of some {\em closed normal}
inclusion of topological groups $\mathfrak{M}\subseteq \mathfrak{X},$ if and only if 
the map $f$ can be equipped with a crossed module structure.
In that case  an essentially unique realization $\mathfrak{M}\subseteq \mathfrak{X}$  
exists by homotopically discrete topological groups. 
\end{abstract}
\maketitle
\section{Introduction}
The purpose of this note is to show that any crossed-module map (any {\it normal}
map in our terminology, see \cite{FS}) is obtained as   the natural map  on the group of
connected components $\pi_0(\frakM)\to \pi_0(\frakX),$ induced
from an inclusion  of a closed normal subgroup
$\frakM\subseteq \frakX$ of a topological group $\frakX$. In addition, the groups $\frakM$ and $\frakX$
can be taken to be {\it homotopically discrete}, i.e.~with contractible connected components.
In what follows we use the term ``normal map" of groups to refer to a map $f:N\to G$ that can be 
equipped with a crossed module structure $G\to \Aut(N).$
\begin{thm}\label{thm main}
For a topological group $\frakG$ let $\pi_0(\frakG)$ denote
the group of connected components of $\frakG,$ that is $\pi_0(\frakG)=\frakG/\frakG_1,$
where $\frakG_1$ is the connected component of the identity of $\frakG$.
Then:
\begin{enumerate}
\item
Suppose that $\frakM$ is a normal subgroup of the topological group $\frakX$.
Then the induced  map $\pi_0(\frakM)\to \pi_0(\frakX)$ given by $\frakM_1\gc\mapsto\frakX_1\gc,\quad\gc\in\frakM$
is a normal map.

\item
Conversely, given a normal map $n\colon N\to G$ of discrete groups, there is a topological
group $\frakX$ and a closed normal subgroup $\frakM$ of $\frakX$ with
a commuting diagram and isomorphisms as indicated:
\[
\xymatrix{
\pi_0(\frakM)\ar[r]^{\bar n}\ar[d]_{\cong} & \pi_0(\frakX)\ar[d]^{\cong}\\
N\ar[r]^n & G\\
}
\]
and with $\bar n$ the naturally induced  map of connected components.
\item In addition, the topological groups $\frakM$ and $\frakX$ in part (2) realizing the map $n$
can be chosen to be homotopically discrete.
\end{enumerate}
\end{thm}

Moreover, given the proofs of the above, it is not hard to see that  for any map of groups $f:\Gamma\to G$ the set
of crossed module structures on $f$ (up to equivalence) is
naturally isomorphic to the set of ``realization" of $f,$ i.e.~maps of homotopically discrete topological groups
$\frakM\subseteq \frakX$ such that $f$ is the induced map    $\pi_0(\frakM)\to\pi_0(\frakX)$ (again up to the obvious 
equivalence of such realizations).  

\setcounter{subsection}{1}
\subsection{An example}
Given a normal map $n: N\to G$ (a crossed module map),
we will call an {\it inclusion} map $\bar n\colon\frakM\to\frakX$ as in Theorem \ref{thm main}(2),
{\it a closed normal realization} of $n$.
One example to keep in mind is the trivial map $\zz\to 0$ from the group of integers to the zero group,
taken as a map with a trivial crossed-module structure.
This map is induced from the closed normal realization $\zz\to \rr$ of the integers into the group
of real numbers; both topological groups here are homotopically discrete.
Thus  our point  is that such a realization  can be made for
{\it every crossed module map} $n: N\to G,$ and it is essentially unique 
with the  given  a crossed module structure on the map $n.$
This is not surprising in light of the well known comment by Quillen
that any normal  map $n$ is of the form $\pi_1(F)\to \pi_1(E)$
for some map  of  connected pointed topological spaces $F\to E,$
which is a part of a fibration sequence of connected spaces $F\to E\to B.$
Taking loop spaces we get a map
of loop spaces $\Omega F\to \Omega E$ whose ``quotient'' is the loop space $\Omega B$.
Thus the given map $n$  is now isomorphic to the map induced on path components
$\pi_0(\Omega F)\to \pi_0(\Omega E).$

 Hence our point in this note is to comment  that  the loop-map of loop spaces
given by Quillen's observation can  always be thought of, and has a closed
normal realization  as an
inclusion of (homotopically discrete) topological groups, as in Theorem \ref{thm main}.
In fact, this realization as a map of topological groups is probably clear
to the experts, since loop-maps of loop spaces can be ``rigidified'' into
maps of topological groups with the same underlying homotopy types.

Here we present  a {\it direct construction} that yields {\it homotopically discrete
spaces} as realizations, by Theorem \ref{thm main}(3).  In fact, it is not hard to show that any other
closed normal realization of $n$  maps uniquely to the homotopically
discrete realization constructed in  part (3) of Theorem \ref{thm main},
i.e., the realization in Theorem \ref{thm main}(3) is the universal-terminal  one.

\subsection{Topological groups and simplicial groups}
The main line of our construction in Theorem \ref{thm main}(2)
uses {\it simplicial groups} in order to
construct the topological groups $\frakX$ and its closed normal subgroup $\frakM$.
Indeed, we use the simplicial
group structure obtained in \cite{FS} on the {\it bar construction} ${\rm Bar}(G,N)$.
One observes that in order to prove
part (2) of Theorem \ref{thm main} it suffices to construct a closed normal realization
of simplicial groups  that induces the given normal map
on path components. The point is that one can realize a simplicial
set $X_{\bu}$ as a topological space $|X_{\bu}|$.
As is well known:
(i)\, realization turns  a simplicial group into a topological group and
(ii)\, realization turns a normal simplicial subgroup into
a closed normal subgroup. These facts follow from
observations of J.~Milnor (see \cite{M}).  The reason for assertion (i)
is that realization
as a functor from simplicial sets to topological spaces
strictly commutes with products via a natural homeomorphism:
$|X_{\bu}\times Y_{\bu}|\cong  |X_{\bu}|\times |Y_{\bu}|$.
Also, (ii) follows from the fact that a short exact sequence of simplicial groups turns, via realization, to a short
exact sequence of topological groups.

%
\section{Proof of part (1) of Theorem \ref{thm main}}
\medskip

\noindent
Recall that as in \cite{FS} we (usually) apply maps on the right.
We first need the following (easy) lemma.

\begin{lemma}\label{lem normal}
Let $G$ be a group and let $K, M$ be normal subgroups of $G$
with $K\le M$.  Let $\gC\nsg G$ be a normal subgroup of $G$ containing $K$ and
suppose that $[\gC,M]\le K$.
Then the map
\[
n\colon \gC/K\to G/M:\quad K\gc\mapsto M\gc,
\]
induced by the canonical
homomorphism $G/K\to G/M,$ is a normal map.
\end{lemma}
\begin{proof}
For $g\in G,$ let $\hat\ell(g)\in\Aut(\gC/K)$ be defined by $K\gc\mapsto K\gc^g$.
Then $\hat\ell(g)\in\Aut(\gC/K)$ because  $\gC$ and $K$ are normal in $G$.
The map $\hat\ell\colon G\to\Aut(\gC/K)$ is a homomorphism.  Since $[\gC,M]\le K,$
it follows that $M\le\ker(\hat\ell)$ and so $\hat\ell$ induces a homomorphism
$\ell\colon G/M\to\Aut(\gC/K)$ defined by
$(K\gc)^{Mg}=K\gc^g$.  We check that $n$ is a normal map.
We have
\[
\left((K\gc)^{Mg}\right)n=(K\gc^g)n=M\gc^g=(M\gc)^{Mg}=((K\gc)n)^{Mg},
\]
for all $\gc\in\gC$ and $g\in G,$ where $(M\gc)^{Mg}$ is conjugation in $G/M$.  Thus (NM1)
(of subsection 2.4 in \cite{FS}) holds.

Next
\[
(K\gc)^{(K\gd)n}=(K\gc)^{M\gd}=K\gc^{\gd}=(K\gc)^{K\gd},
\]
for all $\gc,\gd\in\gC,$ where $(K\gc)^{K\gd}$ is conjugation in $\gC/K$.  Hence (NM2) holds.
\end{proof}

\begin{notation}
For a topological group $\frakG$ we denote by
$\frakG_1$ the connected component of the
identity of $\frakG$.  We let $\pi_0(\frakG)$ be the group of connected components
of $\frakG,$ so $\pi_0(\frakG)=\frakG/\frakG_1$.  Notice that if $\frakG$ is discrete,
then $\pi_0(\frakG)=\frakG$.
\end{notation}

\begin{cor}
Let $\frakX$ be a topological group
and let $\frakM$ be a   normal subgroup of $\frakX$.
Then
\begin{enumerate}
\item
$\frakM_1$ is normal in $\frakX$ and $\frakM_1\le \frakX_1;$

\item
$[\frakX_1,\frakM]\le \frakM_1;$

\item
the map $n\colon\pi_0(\frakM)\to \pi_0(\frakX)$ defined by $\frakM_1\gc\mapsto \frakX_1\gc,$ $\gc\in\frakM,$ is a normal map.
\end{enumerate}
\end{cor}
\begin{proof}
Since $(\frakM_1)a=\frakM_1,$ for any continuous automorphism $a$ of $\frakM,$
and since the restriction of conjugation by $g\in \frakX$ to $\frakM$
is a continuous automorphism of $\frakM,$ it follows that $\frakM_1\nsg \frakX$.
Clearly $\frakM_1\le \frakX_1$.

For each $m\in\frakM$ the map $[\ \, , m]\colon x\mapsto [x,m]=x^{-1}x^m$ is a continuous
map $\frakX\to\frakX$ because it is the product of the two continuous maps:
$x\mapsto x^{-1}$ and $x\mapsto x^m$.  Hence, since $\frakX_1$
is connected, its image under the map $[\ \, ,m]$
is connected and it is contained in $\frakM$.  It
follows that $[\frakX_1,m]\subseteq\frakM_1,$ so $[\frakX_1,\frakM]\le\frakM_1$
and part (2) holds.
Part (3) follows from Lemma \ref{lem normal}.
\end{proof}

\section{Proof of parts (2) and (3) of Theorem \ref{thm main}}
In this section we prove parts (2) and (3) of Theorem \ref{thm main}.
For this we use the results of \cite{FS}.  Our notation is as
in \cite{FS}.

Given a group homomorphism $n\colon N\to G,$
the {\it simplicial set ${\rm Bar}(G,N)$} is discussed in \cite[subsection 2.3]{FS}
(note that the map $n$ is suppressed from the notation).  In section 4 of \cite{FS}
it is shown that when the map $n$
admits a crossed module structure, then ${\rm Bar}(G,N)$ can be equipped with
a {\it simplicial group structure}.  Moreover, in this case the natural map
$G\to {\rm Bar}(G,N)_k$ is a homomorphism, for all $k\ge 0$.  We
then say that ${\rm Bar}(G,N)$ can be equipped with
a {\it normal} simplicial group structure (the interested reader
can look at 
\cite[Definition 2.4.1]{FS}).

In order to prove parts (2) and (3) of Theorem \ref{thm main}
we construct a simplicial normal embedding $N_{\bu}\hookrightarrow X_{\bu},$
where $N_{\bu}$ is a constant simplicial group ($N_k\cong N$) and
$X_{\bu}$ is given in subsection \ref{sub xdot} below.
The simplicial group $X_{\bu}$ is homotopically descrete and equivalent
to $G$.

To show that $N_{\bu}\hookrightarrow X_{\bu}$ is
a normal embedding with the required properties on the
connected components, we present $N_{\bu}$ as a kernel of
a simplicial group map from $X_{\bu}$ to ${\rm Bar}(G,N)$.

The first step in the construction of $X_{\bu}$ is to consider the contractible
simplicial group ${\rm Bar}(N,N)$ in the following lemma.

\begin{lemma}\label{b(n,n)}
Let $N$ be a group.  Then ${\rm Bar}(N,N),$ where the map $n'\colon N\to N$
used to define ${\rm Bar}(N,N)$
is the identity map, with the multiplication on ${\rm Bar}(N,N)_k$
defined by
\[
(a_0,\dots,a_k)*(b_0,\dots,b_k)=(a_0b_0,a_1^{b_0}b_1,a_2^{b_0b_1}b_2,\dots,a_k^{b_0\cdots b_{k-1}}b_k).
\]
gives a normal simplicial group structure on ${\rm Bar}(N,N).$
\end{lemma}
\begin{proof}
Recall that for $k\ge 0,$  ${\rm Bar}(N,N)_k:=\overbrace{N\times\dots\times N}^{(k+1)\text{-times}}$.
Since $n'$ is the identity map, it is a normal map
with the normal structure: conjugation
in $N$.  Hence the lemma follows from \cite[Theorem 4.1]{FS}.
\end{proof}

Throughout the rest of this section $n\colon N\to G$ is a normal map
with the normal structure $\ell\colon G\to\Aut(N),$ as defined in
\cite[subsection 2.5, p.~362]{FS}.  As in \cite{FS} we denote
\[
\ell(g)\colon a\mapsto a^g,\ a\in N\text{ and }g\in G.
\]
\setcounter{subsection}{1}
\subsection{The simplicial group $G\ltimes {\rm Bar}(N,N)$}\label{sub xdot}\hfill
\numberwithin{prop}{subsection}
\setcounter{prop}{0}
\medskip

\noindent
The purpose of this subsection is to define the simplicial group
\[
X_{\bu}=G\ltimes {\rm Bar}(N,N).
\]
Recall that ${\rm Bar}(G,N)$ is endowed with a normal simplicial group
structure by \cite[section 4, p.~365]{FS}

\begin{lemma}\label{lem lk}
The map $\ell_k\colon G\mapsto \Aut({\rm Bar}(N,N)_k)$ defined by
\[\tag{$*$}
\ell_k(g)\colon (a_0,\dots,a_k)\mapsto (a_0^g,\dots,a_k^g)
\]
is a group homomorphism, for all $k\ge 0$.
\end{lemma}
\begin{proof}
By \cite[Lemma 3.5(1)]{FS},
the group $N_k$ defined in \cite[Notation 2.4.4(2)]{FS}
for $k\ge 1$ is isomorphic to ${\rm Bar}(N,N)_{k-1}$.
By \cite[Lemma 3.5(2)]{FS}, the map defined in $(*)$
is conjugation by $(g,1,\dots,1)$ in ${\rm Bar}(G,N)_{k+1},$
hence this map is an automorphism of ${\rm Bar}(N,N)_k$.
Clearly $\ell_k$ is a homomorphism.
\end{proof}

\begin{notation}\label{not X}
We denote by $X_{\bu}$ the following simplicial group
(see Lemma \ref{lem X} below for a proof).  For $k\ge 0$ let $\ell_k\colon G\mapsto \Aut({\rm Bar}(N,N)_k)$
be as in Lemma
\ref{lem lk}.  We let $X_k$ be the semidirect product $X_k=G\ltimes_{\ell_k} {\rm Bar}(N,N)_k$.
We  define the face maps by $d_i(g,(a_0,\dots,a_k))=(g,d_i(a_0,\dots,a_k)),$
$k\ge 1,$
and the degeneracy maps by $s_i(g,(a_0,\dots,a_k))=(g,s_i(a_0,\dots,a_k)),$
$i\ge 0$.
\end{notation}

\begin{lemma}\label{lem X}
Let $X_{\bu}$ be as in notation \ref{not X}.  Then
$X_{\bu}$ is a simplicial group.
\end{lemma}
\begin{proof}
We first show that $d_0$ is a group homomorphism.
Let
\[
(g,(a_0,\dots,a_k)),\ (h,(b_0,\dots,b_k))\in X_k.
\]
Then
\begin{gather*}
(g,(a_0,\dots,a_k))*(h,(b_0,\dots,b_k))=(gh,(a_0^h,\dots,a_k^h)*(b_0,\dots,b_k))\\
=(gh, (a_0^hb_0,\, (a_1^h)^{b_0}b_1,\, (a_2^h)^{b_0b_1}b_2,\,\dots\, ,(a_k^h)^{b_0b_1\cdots b_{k-1}}b_k)).
\end{gather*}
Hence
\begin{gather*}
d_0\big((g,(a_0,\dots,a_k))*(h,(b_0,\dots,b_k))\big)\\
=d_0\big((gh, (a_0^hb_0,\, (a_1^h)^{b_0}b_1,\, (a_2^h)^{b_0b_1}b_2,\,\dots\, ,(a_k^h)^{b_0b_1\cdots b_{k-1}}b_k))\big)\\
=(gh,\, ((a_0a_1)^hb_0b_1,\, (a_2^h)^{b_0b_1}b_2,\,\dots\, ,(a_k^h)^{b_0b_1\cdots b_{k-1}}b_k)).
\end{gather*}
Next,
\[
d_0(g,\,(a_0,\dots,a_k))=(g,(a_0a_1,a_2,\dots,a_k))
\]
and
\[
d_0(h,\,(b_0,\dots,b_k))=(h,(b_0b_1,b_2,\dots,b_k)).
\]
So
\begin{gather*}
d_0(g,\,(a_0,\dots,a_k))*d_0(h,\,(b_0,\dots,b_k))\\
=(g,(a_0a_1,a_2,\dots,a_k))*(h,(b_0b_1,b_2,\dots,b_k))\\
=(gh, ((a_0a_1)^h,a_2^h,\dots,a_k^h)*(b_0b_1,b_2,\dots,b_k))\\
=(gh,\, ((a_0a_1)^hb_0b_1,\, (a_2^h)^{b_0b_1}b_2,\,\dots\, ,(a_k^h)^{b_0b_1\cdots b_{k-1}}b_k)).
\end{gather*}
This shows that $d_0$ is a group homomorphism.  The proof that $d_i$ is
a group homomorphism, for $1\le i <k,$ is identical.

Next we show that $d_k$ is a group homomorphism.
We have
\begin{gather*}
d_k\big((g,(a_0,\dots,a_k))*(h,(b_0,\dots,b_k))\big)\\
=d_k\big((gh, (a_0^hb_0,\, (a_1^h)^{b_0}b_1,\, (a_2^h)^{b_0b_1}b_2,\,\dots\, ,(a_k^h)^{b_0b_1\cdots b_{k-1}}b_k))\big)\\
=(gh,\, (a_0^hb_0,\, (a_1^h)^{b_0}b_1,\,\dots\, ,(a_{k-1}^h)^{b_0b_1\cdots b_{k-2}}b_{k-1}))
\end{gather*}
while
\begin{gather*}
d_k(g,\,(a_0,\dots,a_k))*d_k(h,\,(b_0,\dots,b_k))\\
=(g,(a_0,a_1,\dots,a_{k-1}))*(h,(b_0,b_1,\dots,b_{k-1}))\\
=(gh,\, (a_0^hb_0,\, (a_1^h)^{b_0}b_1,\,\dots\, ,(a_{k-1}^h)^{b_0b_1\cdots b_{k-2}}b_{k-1})).
\end{gather*}

The proof that the degeneracies are group homomorphisms is similar and omitted here.
\end{proof}

\numberwithin{prop}{section}
\setcounter{prop}{2}
\begin{lemma}\label{lem eta}
The map $\eta\colon X_{\bu}\to {\rm Bar}(G,N)$ defined by
\[
\eta\colon (g,(a_0,\dots,a_k)\to (g(a_0n),a_1,\dots, a_k),
\]
is a homomorphism of simplicial groups.
\end{lemma}
\begin{proof}
We first show that $\eta_k\colon X_k\to {\rm Bar}(G,N)_k$ is a group homomorphism.
We have
\begin{gather*}
[(g,(a_0,\dots,a_k))*(h,(b_0,\dots,b_k))]\eta_k=(gh,(a_0^hb_0, (a_1^h)^{b_0}b_1,\dots,(a_k^h)^{b_0\cdots b_{k-1}}b_k))\eta_k\\
=(gh(a_0^hb_0)n,(a_1^h)^{b_0}b_1,\dots,(a_k^h)^{b_0\cdots b_{k-1}}b_k)\\
=(gh(a_0n)^h(b_0n),(a_1^h)^{b_0}b_1,\dots,(a_k^h)^{b_0\cdots b_{k-1}}b_k)\\
=(g(a_0n)h(b_0n),(a_1^h)^{b_0}b_1,\dots,(a_k^h)^{b_0\cdots b_{k-1}}b_k),
\end{gather*}
where we used the fact that $n$ is a normal map,  while
\begin{gather*}
(g,(a_0,\dots,a_k))\eta_k \circ(h,b_0,\dots,b_k)\eta_k=(g(a_0n),a_1,\dots,a_k)\circ(h(b_0n),b_1,\dots,b_k)\\
=(g(a_0n)h(b_0n), a_1^{h(b_0n)}b_1, a_2^{h(b_0n)(b_1)n}b_2,\dots a_k^{h(b_0n)(b_1\cdots b_{k-1})n}b_k)\\
=(g(a_0n)h(b_0n),(a_1^h)^{b_0}b_1,\dots,(a_k^h)^{b_0\cdots b_{k-1}}b_k).
\end{gather*}

Next we show that $\eta$ commutes with the face and degeneracy maps.  That is
$\eta_kd_i=d_i\eta_{k-1},$ for all $k\ge 1$.  And $\eta_ks_i=s_i\eta_{k+1},$ for all $k\ge 0$.

We first check it for $d_0$.  We have
\[
(g,(a_0,\dots,a_k))\eta_kd_0=(g(a_0n),a_1,\dots,a_k)d_0=(g(a_0a_1)n,a_2,\dots,a_k),
\]
while
\begin{gather*}
(g,(a_0,\dots,a_k))d_0\eta_{k-1}=(g,a_0a_1,a_2,\dots,a_k)\eta_{k-1}=(g(a_0a_1)n,a_2,\dots,a_k).
\end{gather*}
For $0<i<k,$ we have
\begin{gather*}
(g,(a_0,\dots,a_k))\eta_kd_i=(g(a_0n),a_1,\dots,a_k)d_i\\
=(g(a_0n),\dots, a_ia_{i+1},a_{i+2},\dots,a_k),
\end{gather*}
while
\begin{gather*}
(g,(a_0,\dots,a_k))d_i\eta_{k-1}=(g,a_0,\dots, a_ia_{i+1},a_{i+2},\dots a_k)\eta_{k-1}\\
=(g(a_0n),\dots, a_ia_{i+1},a_{i+2},\dots,a_k).
\end{gather*}
Also
\begin{gather*}
(g,(a_0,\dots,a_k))\eta_kd_k=(g(a_0n),a_1,\dots,a_k)d_k\\
=(g(a_0n),a_1,\dots,a_{k-1}),
\end{gather*}
while
\begin{gather*}
(g,(a_0,\dots,a_k))d_k\eta_{k-1}=(g,a_0,\dots a_{k-1})\eta_{k-1}\\
=(g(a_0n), a_1,\dots,a_{k-1}).
\end{gather*}

Checking   that $\eta$ commutes with the degeneracy maps is similar and omitted.
\end{proof}

\begin{lemma}\label{lem kereta}
Let $\eta\colon X_{\bu}\to {\rm Bar}(G,N)$ be the morphism of
simplicial groups defined in Lemma \ref{lem eta}.
Then $\ker{\eta}=M_{\bu}$ is the following normal simplicial subgroup of $X_{\bu}:$
\[
M_k=\{(an, (a^{-1},\overbrace{1,\dots,1}^{k\text{-times}}))\mid a\in N\},\quad k\ge 0.
\]
\end{lemma}
\begin{proof}
This is immediate from the definitions.
\end{proof}

Recall that the (path) connected components of a simplicial group $H_{\bu},$
denoted $\pi_0(H_{\bu}),$
are defined as the equivalence classes of $H_0$ modulo the equivalence relation
given by $d_0(y)\tld d_1(y),$ for $y\in H_1$.  The {\it group of connected
components} of $H_{\bu}$ is thus $\pi_0(H_{\bu})=H_0/V$ where $V$ is the connected
component of the identity element in $H_0$.

Further, the geometric realization
of $H_{\bu},$ which we denote $|H_{\bu}|,$ is a topological group
with the same group of connected components as that of $H_{\bu},$
that is $\pi_0(H_{\bu})=\pi_0(|H_{\bu}|)$.

\begin{remark}
Note that $M_{\bu}$ of Lemma \ref{lem kereta} is discrete.
Indeed, if $(an,(a^{-1},1))\in M_1,$ then $d_0((an,(a^{-1},1)))=(an,a^{-1})=d_1((an,(a^{-1},1))).$
\end{remark}

Recall that for a topological group $\frakG$ we denote by $\frakG_1$
the connected components of the identity of $\frakG$.

\begin{lemma}\label{lem thmmain2}
Let $\frakX:=|X_{\bu}|$ and $\frakM:=|M_{\bu}|$.
\begin{enumerate}
\item
$\frakX_1=\{(1,a))\mid a\in N\}\cong N$.
Hence
\[
\pi_0(\frakX)=X_0/\frakX_1=(G\ltimes_{\ell} N)/\frakX_1\cong G.
\]
\item
$\pi_0(\frakM)=M_0=\{(an,a^{-1})\mid a\in N\}\cong N.$

\item
Let $\iota\colon \frakM\to \frakX$ be the inclusion map and let
$\bar n\colon\pi_0(\frakM)\to\pi_0(\frakX)$ be the map
induced from $\iota$.  Then
we have the following commutative diagram
\[
\xymatrix{
\pi_0(\frakM)\ar[r]^{\bar{n}}\ar[d]_{(an,a^{-1})\mapsto a} & \pi_0(\frakX)\ar[d]^{(g,1)\frakX_1\mapsto g}\\
N\ar[r]^n                   &G
}
\]
where the vertical arrows are isomorphisms.
\end{enumerate}
\end{lemma}
\begin{proof}
For part (1), let $(g,a)\in X_0$.  Then $(g,a)$ is in the same connected
component of $(1,1)$ iff there exists $(h,(a_0,a_1))\in X_1,$ such
that $d_0((h,(a_0,a_1))=(1,1)$ and $d_1((h,(a_0,a_1)))=(g,a)$.
Hence $(h,a_0a_1)=(1,1)$ and $(h,a_0)=(g,a)$.  Thus $h=1$ and $a_1=a_0^{-1}$.
Parts (2) and (3) are obvious.
\end{proof}

Note now that Lemma \ref{lem thmmain2} is part (2) of Theorem \ref{thm main}.
Part (3) follows immediately from the constructions above. Since ${\rm Bar}(N,N)$ is contractible,  the map of simplicial groups
is by construction a  map of homotopically discrete spaces. \qed


\end{document}